\documentclass[leqno,a4paper]{amsart}
\usepackage[T1]{fontenc}
\usepackage[utf8]{inputenc}
\usepackage{amsmath,amssymb,amsfonts,amsthm,mathrsfs,mathtools}
\usepackage{enumitem}
\usepackage{xcolor}
\usepackage[colorlinks,linkcolor=blue,citecolor=green!45!black,urlcolor=blue]{hyperref}
\numberwithin{equation}{section}
\newtheorem{theorem}{Theorem}[section]
\newtheorem{proposition}[theorem]{Proposition}
\newtheorem{corollary}[theorem]{Corollary}
\newtheorem{lemma}[theorem]{Lemma}
\theoremstyle{definition}
\newtheorem{definition}[theorem]{Definition}
\newtheorem{example}[theorem]{Example}
\theoremstyle{remark}
\newtheorem{remark}[theorem]{Remark}
\newtheorem{problem}[theorem]{Problem}
\newcommand{\com}{\mathbb C}\newcommand{\ent}{\mathbb Z}
\newcommand{\rat}{\mathbb Q}\newcommand{\real}{\mathbb R}
\newcommand{\FF}{\mathbb F}\newcommand{\T}{\mathbb T}
\newcommand{\bM}{\mathbb M}\newcommand{\cU}{\mathcal U}
\newcommand{\M}{\mathcal M}\newcommand{\N}{\mathcal N}
\newcommand{\R}{\mathcal R}\renewcommand{\L}{\mathcal L}
\newcommand{\E}{\mathcal E}\newcommand{\A}{\mathcal A}\newcommand{\Pcal}{\mathcal P}
\newcommand{\id}{\mathrm{Id}}\newcommand{\cb}{\mathrm{cb}}
\newcommand{\barot}{\,\overline\otimes\,}
\newcommand{\normtwo}[1]{\lVert#1\rVert_2}
\DeclareMathOperator{\EL}{EL}\DeclareMathOperator{\SL}{SL}
\DeclareMathOperator{\Ad}{Ad}\DeclareMathOperator{\tr}{tr}
\DeclareMathOperator{\Tr}{Tr}

\title[QWEP and twisted crossed products]{QWEP stability under twisted crossed products}
\author{Kai Zeng}
\address{Department of Mathematics: Analysis, Logic and Discrete Mathematics, Ghent University, Belgium}
\email{kai.zeng@ugent.be}
\date{}
\hypersetup{pdftitle={QWEP stability under twisted crossed products},pdfauthor={Kai Zeng}}
\begin{document}
\begin{abstract}
We study the QWEP property of twisted crossed products and its relation to Fourier multiplier transference on noncommutative $L_p$-spaces. We show that QWEP is not preserved by crossed products of QWEP von Neumann algebras by hyperlinear groups. Our examples include trace-preserving actions of a torsion-free residually finite Kazhdan group on a diffuse abelian algebra and on the hyperfinite type II$_1$ factor. We also obtain non-QWEP crossed products with nontrivial scalar cocycles. For scalar twists, we prove that a finite trace-preserving crossed product is QWEP if and only if its ordinary crossed product and its twisted group von Neumann algebra are QWEP. In the positive direction, we establish a restriction theorem for scalar twists of regular wreath products. These results provide QWEP inputs for Fourier multiplier transference and identify a completely bounded multiplier norm obstruction to QWEP. Finally, operator-valued cocycles give finite QWEP crossed products whose acting groups are not hyperlinear.
\end{abstract}
\maketitle

\section{Introduction}
This paper is motivated by our work on Fourier multipliers on twisted crossed products; see \cite{xxz26,Zeng23,Zeng26}. In \cite[Theorems~3.2--3.3 and~4.1--4.2]{xxz26}, we established transference results between Fourier multipliers on group von Neumann algebras, Fourier multipliers on twisted crossed products, and Schur multipliers on Schatten classes. These results extend the corresponding theorems for discrete groups and quantum tori in \cite[Theorem~2.7\textup{(b)}]{NR11} and \cite[Theorem~7.3 and Proposition~7.8]{CXY12}. The QWEP property enters the proofs through the tensor extension of completely bounded maps, whereas amenability of the action provides the approximating maps used in Schur--Fourier transference; see also \cite[ Theorem~A]{Gonz}. We recall the definition of QWEP. A unital $C^*$-algebra $\A\subset B(H)$ has the WEP (weak expectation property) if there is a unital completely positive map $P:B(H)\to\A^{**}$ fixing $\A$. A $C^*$-algebra $\mathcal B$ is QWEP if it is a quotient of a $C^*$-algebra with WEP. For a countable discrete group $G$, $\L(G)$ is QWEP if and only if $G$ is hyperlinear. This means that $G$ embeds into a metric ultraproduct of finite-dimensional unitary groups with their normalized Hilbert--Schmidt metrics; see \cite[Definition~3.2]{Pestov08}.

Let us describe the QWEP assumption more precisely. Let $1<p<\infty$. If $\M$ is a finite QWEP von Neumann algebra and $T$ is completely bounded on $L_p(\L(G))$, then $\id\otimes T$ extends to $L_p(\M\barot\L(G))$ with the same completely bounded norm. For a finite trace-preserving twisted crossed product $\R=\M\rtimes_{\alpha,\sigma}G$, the canonical map
\[
x\lambda_\sigma(g)\longmapsto x\lambda_\sigma(g)\otimes\lambda(g)
\]
is a trace-preserving embedding into $\R\barot\L(G)$. If $\R$ is QWEP and $T=T_\varphi$ is a Fourier multiplier, restricting the tensor extension to this copy of $L_p(\R)$ gives a Fourier multiplier estimate on the crossed product. This leads naturally to the permanence question raised in \cite[Problems~2.22 and~2.26]{xxz26} and \cite[Problem 3.1.21]{Zeng23}: if $\M$ is QWEP and $G$ is hyperlinear, must $\M\rtimes_{\alpha,\sigma} G$ be QWEP?

We give a negative answer to this question. Our construction uses the internality theorem of Alekseev--Liu--Thom \cite[Theorem~B]{ALT26} and the normalization theorem of Thom \cite[Theorem~1.2]{Thom26}. Starting from the polynomial and Laurent polynomial construction of Kun--Thom \cite[Theorem~E]{KT26}, we obtain a torsion-free residually finite Kazhdan group with the following properties.

\begin{theorem}\label{thm:overview}
There is a finitely generated, torsion-free, residually finite Kazhdan group $\Gamma_0$ such that:
\begin{enumerate}[label=\textup{(\roman*)},leftmargin=2em]
\item There are trace-preserving actions of $\Gamma_0$ on a diffuse abelian von Neumann algebra $\M$, on $\L(\FF_\infty)$, and on the hyperfinite type II$_1$ factor $\R_0$, whose crossed products are not QWEP. The action on $\R_0$ is outer.
\item For $G=\Gamma_0\times\ent$, there are a trace-preserving action $\beta$ on $\M$ and a scalar cocycle $\sigma$ such that $\M$ and $\L_\sigma(G)$ are QWEP, while $\M\rtimes_{\beta,\sigma}G$ is not QWEP. The cocycle cannot be removed by a $\cU(\M)$-valued cochain.
\item There are finitely generated torsion-free nonhyperlinear groups $E$ and $\Delta$ with split exact sequences
\begin{gather*}
1\longrightarrow A\longrightarrow E\longrightarrow\Gamma_0\longrightarrow1,\\
1\longrightarrow \FF_\infty\longrightarrow\Delta\longrightarrow\Gamma_0\longrightarrow1,
\end{gather*}
where $A$ is free abelian of countably infinite rank and $\M=\L(A)$.
\item There is a finite QWEP von Neumann algebra $\Pcal$ and a trace-preserving cocycle action $(\alpha,\sigma)$ of $E$ on $\Pcal$ such that $\Pcal\rtimes_{\alpha,\sigma}E$ is QWEP. This cocycle action cannot be changed into a scalar-twisted action by perturbing its implementing unitaries with unitaries in $\Pcal$.
\end{enumerate}
\end{theorem}

The hyperfinite-factor example also answers the permanence question for residually finite acting groups discussed by Peterson in \cite[final paragraph]{Peterson10}. The distinction between scalar and operator-valued cocycles in the theorem is explained by the following criterion.

\begin{theorem}\label{thm:positiveoverview}
Let $\alpha:G\curvearrowright(\M,\tau)$ be a trace-preserving action on a finite von Neumann algebra and let $\sigma\in Z^2(G,\T)$. Then
\[
\M\rtimes_{\alpha,\sigma}G\text{ is QWEP}
\quad\Longleftrightarrow\quad
\M\rtimes_\alpha G\text{ and }\L_\sigma(G)\text{ are QWEP}.
\]
More generally, the equivalence holds whenever $\alpha$ admits an invariant normal state.
\end{theorem}

Thus the scalar-twisted crossed products associated with the actions in Theorem~\ref{thm:overview}\textup{(i)} remain non-QWEP. The QWEP problem for $\L_\sigma(G)$ is equivalent to a prescribed projective matrix approximation problem, and it also describes QWEP for crossed products with type-I coefficients. The positive results of Chakraborty--Flores \cite[Theorems~A and~B]{CF26} apply to several classes of groups. We obtain a further permanence result for regular wreath products.

\begin{theorem}\label{thm:wreathoverview}
Let $A$ and $\Gamma$ be countable groups, with $\Gamma$ residually finite, and let $W=A^{(\Gamma)}\rtimes\Gamma$. For $\sigma\in Z^2(W,\T)$, put
\[
\alpha=\sigma|_{A_e\times A_e},\qquad \eta=\sigma|_{\Gamma\times\Gamma}.
\]
Then $\L_\sigma(W)$ is QWEP if and only if both $\L_\alpha(A)$ and $\L_\eta(\Gamma)$ are QWEP.
\end{theorem}

For Fourier multipliers, these results have two consequences. The positive QWEP statements allow us to apply the tensor extension theorem and obtain cocycle-independent completely bounded multiplier norms. Conversely, a strict increase of a completely bounded Fourier multiplier norm under passage to a twisted crossed product implies that this crossed product is not QWEP. We formulate these statements in Section~\ref{sec:fourier}, together with the Schur multiplier comparison under amenability of the action. This gives an analytic approach to the scalar-twist question.

The paper is organized as follows. Section~\ref{sec:input} recalls the required facts about twisted crossed products and QWEP. Section~\ref{sec:pair} constructs the groups used in the counterexamples, and Section~\ref{sec:ordinary} proves the negative permanence results. Sections~\ref{sec:twisted} and~\ref{sec:typeI} treat scalar cocycles and type-I coefficients. The regular wreath-product theorem is proved in Section~\ref{sec:regularwreath}. We discuss Fourier multiplier transference in Section~\ref{sec:fourier}. Finally, Section~\ref{sec:reverse} gives the reverse examples for operator-valued cocycles, including the measured-groupoid construction.

\section{Preliminaries}\label{sec:input}
All groups are countable and discrete. Unless otherwise stated, finite von Neumann algebras have separable predual and are equipped with a faithful normal tracial state. We write $\L(G)$ for the group von Neumann algebra and $\L_\sigma(G)$ for its scalar-twisted version, with the convention $(\lambda_\sigma(s)\xi)(t)=\sigma(t^{-1},s)\xi(s^{-1}t)$. All crossed products are taken in their regular von Neumann algebra representations.

\begin{definition}[{\cite{xxz26}, \cite{Zeng23}}]
A normalized cocycle action \((\alpha,\sigma)\) of \(G\) on \(\M\) consists of automorphisms \(\alpha_g\) and unitaries \(\sigma(g,h)\in \M\) such that
\begin{align*}
\alpha_g\alpha_h&=\Ad_{\sigma(g,h)}\alpha_{gh},\\
\sigma(g,h)\sigma(gh,k)&=\alpha_g(\sigma(h,k))\sigma(g,hk),
\end{align*}
with $\alpha_e=\id$ and $\sigma(e,g)=\sigma(g,e)=1$. Fix a faithful normal representation $\M\subset B(H)$. On $\ell_2(G,H)$, put
\[
(\pi_\alpha(a)\xi)(t)=\alpha_{t^{-1}}(a)\xi(t),\qquad
(\lambda_\sigma(s)\xi)(t)=\sigma(t^{-1},s)\xi(s^{-1}t).
\]
We identify $\M$ with $\pi_\alpha(\M)$. The crossed product is generated by \(\M\) and unitaries \(\lambda_\sigma(g)\) with
\[
\lambda_\sigma(g)x\lambda_\sigma(g)^*=\alpha_g(x),\qquad \lambda_\sigma(g)\lambda_\sigma(h)=\sigma(g,h)\lambda_\sigma(gh).
\]
A scalar cocycle has \(\sigma(g,h)\in\T1\); then \(\alpha\) is an ordinary action. If all \(\alpha_g\) preserve \(\tau\), the canonical trace satisfies \(\widehat\tau(x\lambda_\sigma(g))=0\) for \(g\neq e\), and equals \(\tau(x)\) for \(g=e\).
\end{definition}

For a nonprincipal ultrafilter \(\omega\), write
\[
\mathcal Q_\omega=\prod_{n\to\omega}(\bM_{d_n},\tr_{d_n}),
\qquad \tr_d=d^{-1}\Tr.
\]
This quotient uses operator-norm bounded sequences and the ideal of sequences whose normalized \(2\)-norm tends to zero along $\omega$. A subalgebra is \emph{internal} if it is \([B_n]_\omega\) for unital \(*\)-subalgebras \(B_n\subset \bM_{d_n}\).

\begin{theorem}[Alekseev--Liu--Thom, {\cite[Theorem~B]{ALT26}}]
If a finite unitary tuple in \(\mathcal Q_\omega\) has conjugation spectral gap off its relative commutant, that commutant is internal. Consequently, for every Kazhdan group \(H\) and every homomorphism \(\pi:H\to\cU(\mathcal Q_\omega)\), the algebra \(\pi(H)'\cap\mathcal Q_\omega\) is internal.
\end{theorem}
For the consequence, apply property \((T)\) to the conjugation representation on \(L_2(\mathcal Q_\omega)\).

For \(H\leq G\), put
\[
P_H=\{t\in G:tHt^{-1}\subseteq H\}.
\]
The subgroup \(H\) is \emph{infranormal} if this monoid generates \(G\) as a group.

\begin{theorem}[Thom, {\cite[Theorem 1.2]{Thom26}}]\label{thm:normalization}
If \(H\leq G\) are Kazhdan and \(H\) is infranormal, then every homomorphism \(\pi:G\to\cU(\mathcal Q_\omega)\) satisfies
\[
\pi(g)(\pi(H)'\cap\mathcal Q_\omega)\pi(g)^*=\pi(H)'\cap\mathcal Q_\omega\qquad(g\in G).
\]
\end{theorem}
The internality theorem supplies the centralizer hypothesis in \cite[Theorem 1.2]{Thom26}, and hence gives this normalization statement.

QWEP passes to expected subalgebras and to spatial tensor products by \cite[Proposition~4.1\textup{(ii),(v)}]{Ozawa04}. It also passes to complex conjugate algebras. For countable groups, \cite[Corollary~6.2 and Proposition~7.1]{Ozawa04} give
\begin{equation}\label{eq:bridge}
G\text{ hyperlinear}\quad\Longleftrightarrow\quad
(\L(G),\tau_G)\text{ Connes embeddable}\quad\Longleftrightarrow\quad
\L(G)\text{ QWEP}.
\end{equation}
In particular, residually finite groups have QWEP group von Neumann algebras.

\begin{definition}[Chakraborty--Flores, {\cite[Definition~2.1]{CF26}}]
A countable group \(G\) is \emph{cocycle-hyperlinear} if \(\L_\sigma(G)\) is Connes embeddable for every normalized scalar \(2\)-cocycle \(\sigma\). For the finite tracial algebras here, Connes embeddability and QWEP are equivalent. In particular, a cocycle-hyperlinear group is hyperlinear.
\end{definition}
Cocycle-hyperlinearity is equivalent to hyperlinearity of every central extension of \(G\) by a countable abelian group; see \cite[Lemma~3.4]{Thom10}.

\section{A torsion-free Kazhdan construction}\label{sec:pair}
Following \cite[Theorem E]{KT26}, set
\[
R_+=\ent[x_1,x_2,x_3],\qquad R=\ent[x_1^{\pm1},x_2^{\pm1},x_3^{\pm1}],
\]
\[
H_{\ent}=\EL_3(R_+),\qquad N_{\ent}=\EL_3(R),\qquad
\Gamma_{\ent}=N_{\ent}\rtimes\SL_3(\ent),
\]
where \(S\in\SL_3(\ent)\) acts by \(x^v\mapsto x^{Sv}\). Evaluation at \((1,1,1)\) gives a homomorphism
\[
\Psi:\Gamma_{\ent}\longrightarrow\SL_3(\FF_3)\times\SL_3(\FF_3),
\qquad (g,S)\longmapsto(g(1,1,1)\bmod3,S\bmod3).
\]
Put
\begin{equation}\label{eq:pair}
\Gamma_0=\ker\Psi,\qquad H_0=H_{\ent}\cap\Gamma_0.
\end{equation}

\begin{proposition}\label{prop:pairprops}
The groups \(\Gamma_0,H_0\) are finitely generated, residually finite, torsion-free Kazhdan groups.
\end{proposition}
\begin{proof}
By \cite[Theorem~1.1]{EJZ10}, \(H_{\ent},N_{\ent}\) and \(\SL_3(\ent)\) have property \((T)\). Stability under extensions and finite-index subgroups gives property \((T)\), hence finite generation, for \(\Gamma_0,H_0\).

The finite ring maps
\[
R\longrightarrow(\ent/m\ent)[(\ent/k\ent)^3]
\]
give finite semidirect-product quotients of \(\Gamma_{\ent}\). If \(S\ne I\), choose \(k\) detecting \(S\). Otherwise, choose \(k\) separating the exponent vectors in a nonzero entry of \(g-I\), and \(m\) retaining a nonzero coefficient. Thus \(\Gamma_{\ent}\), and its subgroups, are residually finite.

An integral matrix of finite order congruent to \(I\) modulo \(3\) equals \(I\). Otherwise a nontrivial prime-order power would satisfy
\[
0=(I+3^aB)^p-I
=p3^aB+\sum_{j=2}^p\binom pj3^{aj}B^j,
\qquad a\geq1,\quad B\not\equiv0\pmod3,
\]
where the first term has strictly smaller minimum \(3\)-adic valuation than the remaining terms.

If \((g,S)\in\Gamma_0\) has finite order, it follows that \(S=I\) and \(g(1,1,1)=I\). Under the injection
\[
\rat[x_1^{\pm1},x_2^{\pm1},x_3^{\pm1}]
\hookrightarrow\rat[[y_1,y_2,y_3]],\qquad x_i\mapsto1+y_i,
\]
a first nonzero homogeneous term \(C_d\) of \(g-I\) would give
\[
g^m=I+mC_d+\text{terms of degree greater than }d\ne I.
\]
Hence \(g=I\), proving torsion-freeness.
\end{proof}

\begin{lemma}\label{lem:monoid}
If a submonoid \(P\subset G\) generates \(G\), and \(G_0\triangleleft G\) has finite index, then \(P\cap G_0\) generates \(G_0\).
\end{lemma}
\begin{proof}
The image of \(P\) is the whole finite quotient. Choose coset representatives in \(P\). A Schreier generator, or its inverse, has the form \(rp(r')^{-1}\), with \(r,p,r'\in P\) and \(rpG_0=r'G_0\). Choose \(c\in P\) representing \((r'G_0)^{-1}\). Then
\[
rp(r')^{-1}=(rpc)(r'c)^{-1},\qquad rpc,r'c\in P\cap G_0.
\]
\end{proof}

\begin{proposition}\label{prop:compression}
The subgroup \(H_0\) is infranormal and nonnormal in \(\Gamma_0\). With \(E_{ij}\) the matrix units, put
\begin{equation}\label{eq:tr}
t=(I_3,I_3+3E_{12}),\qquad r=(I_3+3x_2E_{12},I_3).
\end{equation}
Then \(t\in\Gamma_0\), \(r\in H_0\), and
\begin{equation}\label{eq:strict}
tH_0t^{-1}\subset H_0,\qquad
t^{-1}rt=(I_3+3x_1^{-3}x_2E_{12},I_3)\notin H_0.
\end{equation}
Moreover \(I=\Gamma_0/H_0\) is countably infinite.
\end{proposition}
\begin{proof}
The compression monoid of \(H_{\ent}\) contains \(H_{\ent}\) and \((I_3,I_3+E_{ij})\), \(i\ne j\). The latter preserve nonnegative exponents and generate the substitution group. For \(v=ku\in\ent^3\setminus\{0\}\), with \(k\geq1\) and \(u\) primitive, choose \(S\in\SL_3(\ent)\) with \(Se_1=u\). Then
\[
S\cdot x_1^k=x^v.
\]
Thus substitution conjugates of \(H_{\ent}\) contain all elementary Laurent monomials and generate \(N_{\ent}\). Lemma~\ref{lem:monoid}, together with normality of \(\Gamma_0\), proves infranormality of \(H_0\).

The substitution \(t\) sends \(x_2\) to \(x_1^3x_2\) and fixes \(x_1,x_3\), giving \eqref{eq:strict}. The powers \(t^j\) have distinct substitution components, whereas every element of \(H_0\) has trivial substitution component. Hence their cosets are distinct.
\end{proof}

\subsection{Two split extensions}\label{sec:groups}
Write \([x,y]=xyx^{-1}y^{-1}\), and put
\[
A=\bigoplus_I\ent,\qquad E=A\rtimes\Gamma_0.
\]
Let \(a_x\) generate the coordinate at \(x\in I\), and set \(a=a_{H_0}\). The action is given by \(ga_xg^{-1}=a_{gx}\).

\begin{proposition}[Theorem {\ref{thm:overview}}, (iii)]\label{prop:wreath}
The group \(E\) is finitely generated, torsion-free and nonhyperlinear. The nonidentity word
\[
w=[r,tat^{-1}]
\]
is killed by every homomorphism from \(E\) to the unitary group of a tracial matrix ultraproduct. Its kernel \(A\) and quotient \(\Gamma_0\) in the defining split extension are residually finite.
\end{proposition}
\begin{proof}
By \eqref{eq:strict},
\[
w=a_{rtH_0}a_{tH_0}^{-1}\ne1.
\]
For \(\pi:E\to\cU(\mathcal Q_\omega)\), set \(\mathcal D=\pi(H_0)'\cap\mathcal Q_\omega\). Since \([a,H_0]=1\), Theorem~\ref{thm:normalization} gives
\[
\pi(a)\in \mathcal D,\qquad
\pi(tat^{-1})\in\pi(t)\mathcal D\pi(t)^*=\mathcal D,\qquad
\pi(w)=1.
\]
Thus \(E\) is not hyperlinear. A finite generating set of \(\Gamma_0\), together with \(a\), generates \(E\). Both kernel and quotient are torsion-free, so is \(E\). Finally, coordinate reduction modulo integers separates the elements of \(A\).
\end{proof}

Let \(\Delta=\Gamma_0*_{H_0}\Gamma_0\), with vertex maps \(\iota_1,\iota_2\), and identify their common copy of \(H_0\). Put
\[
a_\Delta=\iota_1(t)^{-1}\iota_2(t),\qquad
b_\Delta=\iota_2(t)\iota_1(t)^{-1},\qquad
w_\Delta=[b_\Delta,r].
\]

\begin{proposition}[Theorem {\ref{thm:overview}}, (iii)]\label{prop:double}
The group \(\Delta\) is finitely generated, torsion-free and nonhyperlinear. Every matrix-ultraproduct homomorphism kills \(w_\Delta\ne1\). The fold map gives a split exact sequence
\[
1\longrightarrow \FF_\infty\longrightarrow\Delta\longrightarrow\Gamma_0\longrightarrow1.
\]
Its kernel and quotient are residually finite.
\end{proposition}
\begin{proof}
Since \(tH_0t^{-1}\subset H_0\), the amalgam relations give \([a_\Delta,H_0]=1\). Moreover,
\[
w_\Delta=\iota_2(t)\,\iota_1(t^{-1}rt)\,\iota_2(t^{-1}r^{-1})\ne1:
\]
the three syllables lie outside \(H_0\), by \eqref{eq:strict} and their substitution components. Apply \cite[Chapter~I, Section~1.2, Theorem~2]{Serre80}. For \(\pi:\Delta\to\cU(\mathcal Q_\omega)\), set \(\mathcal D=\pi(H_0)'\cap\mathcal Q_\omega\). Theorem~\ref{thm:normalization} gives
\[
\pi(a_\Delta)\in \mathcal D,\qquad
\pi(b_\Delta)=\pi(\iota_1(t))\pi(a_\Delta)\pi(\iota_1(t))^*\in \mathcal D,
\qquad \pi(w_\Delta)=1.
\]

The fold map splits through either vertex. Its kernel acts freely on the Bass--Serre tree; the quotient graph has two vertices and one edge for each element of \(\Gamma_0/H_0\). Thus the kernel is \(\FF_\infty\), hence residually finite. Every torsion element is conjugate into a vertex group, so \(\Delta\) is torsion-free; the two vertices also give finite generation. See \cite[Chapter~I, Theorem~4', Proposition~18 and Theorem~8]{Serre80} for these tree-action facts.
\end{proof}

The nonhyperlinearity of \(\Delta\) also follows from \cite[Corollary C]{ALT26}.

\begin{corollary}\label{cor:nonqwep}
Both \(\L(E)\) and \(\L(\Delta)\) are non-QWEP, although all four kernels and quotients displayed above are hyperlinear and sofic.
\end{corollary}
\begin{proof}
Apply \eqref{eq:bridge} and residual finiteness of the kernels and quotients.
\end{proof}

\section{Failure of QWEP stability}\label{sec:ordinary}
The trace-preserving identification
\begin{equation}\label{eq:semidirect}
\L(K\rtimes G)\cong\L(K)\rtimes G
\end{equation}
applied to the preceding split extensions gives
\[
\L(E)\cong \M\rtimes\Gamma_0,\qquad
\M=\L(A)\cong L_\infty(\T^I,m^I),\qquad
\L(\Delta)\cong\L(\FF_\infty)\rtimes\Gamma_0.
\]
Both coefficient algebras and \(\L(\Gamma_0)\) are QWEP, whereas the crossed products are non-QWEP by Corollary~\ref{cor:nonqwep}. Denote the induced action on $\M$ by $\beta$.

\begin{theorem}[Theorem {\ref{thm:overview}}, (i)]\label{thm:hyperfinite}
The group \(\Gamma_0\) admits an outer trace-preserving action on the hyperfinite II$_1$ factor \(\R_0\) whose crossed product is a non-QWEP II$_1$ factor.
\end{theorem}
\begin{proof}
Choose a Haar unitary \(z\in \R_0\), and let \(\E_0:\R_0\to B_0=W^*(z)\) be the trace-preserving expectation. Then
\[
\M\cong\barot_{i\in I}B_0\subset
\R_I=\barot_{i\in I}(\R_0,\tau_0)\cong \R_0.
\]
The coordinate-permutation action \(\widetilde\beta\) extends \(\beta\), and \(\E_\M=\barot_I \E_0\) is equivariant. Hence
\[
\E_\M\rtimes\id:\R_I\rtimes_{\widetilde\beta}\Gamma_0
\longrightarrow \M\rtimes_\beta\Gamma_0,\qquad
\sum_gx_g\lambda(g)\longmapsto\sum_g\E_\M(x_g)\lambda(g)
\]
is a trace-preserving conditional expectation. Its range is \(\L(E)\), so the domain is non-QWEP. The following lemma proves outerness and factoriality.
\end{proof}

\begin{lemma}\label{lem:faithful-outer}
The coset action \(\Gamma_0\curvearrowright I\) is faithful, and its permutation action on \(\R_I\) is outer. Consequently \(\R_I\rtimes\Gamma_0\) is a II$_1$ factor.
\end{lemma}
\begin{proof}
Let \(C=\bigcap_{\gamma\in\Gamma_0}\gamma H_0\gamma^{-1}\). For \((B,I_3)\in C\), normality implies that every entry of \(B\) remains polynomial under
\[
x^v\longmapsto x^{(I_3+3kE_{ij})v},\qquad k\in\ent,\quad i\ne j.
\]
If a nonconstant monomial \(x^v\) occurs, choose \(v_j\ne0\), \(i\ne j\), and \(k\) with \(v_i+3kv_j<0\). Distinct monomials remain distinct, contradicting polynomiality. Thus \(B\) is constant.

For \(i\ne j\), put \(q_{ij}=I_3+3x_1^{-1}E_{ij}\). Since \((q_{ij},I_3)\in\Gamma_0\),
\[
q_{ij}Bq_{ij}^{-1}
=B+3x_1^{-1}(E_{ij}B-BE_{ij})-9x_1^{-2}E_{ij}BE_{ij}
\]
is polynomial. Hence \([E_{ij},B]=0\) for all \(i\ne j\), so \(B=cI_3\). As \(c\in\ent\) and \(\det B=1\), we have \(B=I_3\) and \(C=\{1\}\).

Write \(\R_0=\barot_{n\geq1}(\bM_2,\tr_2)\). For \(g\ne e\), choose \(i\in I\) with \(gi\ne i\), and let \(z_n=\operatorname{diag}(1,-1)\) in the \(n\)-th matrix factor of coordinate \(i\). Then
\[
\normtwo{[x,z_n]}\longrightarrow0\quad(x\in \R_I),\qquad
\normtwo{\widetilde\beta_g(z_n)-z_n}^2=2.
\]
Thus \(\widetilde\beta_g\) is outer. If \(\sum_gx_g\lambda(g)\) is central, then
\[
x_g\widetilde\beta_g(y)=yx_g\qquad(y\in \R_I).
\]
Outerness gives \(x_g=0\) for \(g\ne e\), and \(x_e\in\com\). The crossed product is therefore a finite diffuse factor.
\end{proof}

%\begin{proposition}[A necessary condition in the finite setting]\label{prop:necessary}
%Let \((\M,\tau)\) be a finite von Neumann algebra with faithful normal tracial state, and let \(G\) act preserving \(\tau\). If \(\M\rtimes G\) is QWEP, then both \(\M\) and \(\L(G)\) are QWEP.
%\end{proposition}
%\begin{proof}
%The canonical trace gives conditional expectations onto \(\M\) and \(\L(G)\), respectively:
%\[
%\sum_gx_g\lambda(g)\longmapsto x_e,
%\qquad
%\sum_gx_g\lambda(g)\longmapsto\sum_g\tau(x_g)\lambda(g).
%\]
%Apply QWEP permanence under conditional expectations.
%\end{proof}

\subsection{A mixed scalar cocycle}\label{sec:mixed}
Let \(\M=\L(A)\) and \(G=\Gamma_0\times\ent\), with \(\beta_{(g,m)}=\beta_g\). For \(g=(B,S)\in\Gamma_0\), put
\[
\chi(g)=\zeta^{S_{12}/3},\qquad \zeta=e^{2\pi i/3}.
\]
This is a character, since
\[
S=I_3+3X,\quad T=I_3+3Y
\quad\Longrightarrow\quad (ST)_{12}/3\equiv X_{12}+Y_{12}\pmod3.
\]
Moreover \(\chi(t)=\zeta\). Define the normalized scalar cocycle
\begin{equation}\label{eq:mixedsigma}
\sigma((g,m),(h,n))=\chi(g)^n1_\M.
\end{equation}

\begin{theorem}[Theorem {\ref{thm:overview}}, (ii)]\label{thm:mixed}
Both \(\M\) and \(\L(G)\) are QWEP, whereas
\[
\R_\sigma=\M\rtimes_{\beta,\sigma}G
\]
is non-QWEP. The scalar cocycle \(\sigma\) cannot be removed by a cochain with values in \(\cU(\M)\).
\end{theorem}
\begin{proof}
The first assertion follows from commutativity of \(\M\) and residual finiteness of \(G\). Since \(\sigma|_{\Gamma_0\times\Gamma_0}=1\), there is a trace-preserving expected inclusion
\[
\L(E)\cong \M\rtimes\Gamma_0\subset \R_\sigma.
\]
Hence \(\R_\sigma\) is non-QWEP.

Put \(U_g=\lambda_\sigma(g,0)\) and \(V=\lambda_\sigma(e,1)\). Then \(U_gV=\chi(g)VU_g\). If \((b_x)_{x\in G}\subset\cU(\M)\) removed the cocycle, the corrected implementers for $(g,0)$ and $(e,1)$ would commute. Since \(\M\) is abelian and \(\ent\) acts trivially,
\begin{equation}\label{eq:eigenunitary}
\beta_g(b_{(e,1)})=\chi(g)^{-1}b_{(e,1)}.
\end{equation}
Write \(b_{(e,1)}=\sum_{a\in A}c_a\lambda(a)\) in \(L_2(\M)\). Then
\[
|c_{ga}|=|c_a|,\qquad \sum_{a\in A}|c_a|^2<\infty.
\]
Every \(a\ne0\) has infinite orbit: otherwise the union of its orbit supports would be a finite nonempty invariant subset of the infinite transitive set \(I\). Thus \(c_a=0\) for \(a\ne0\). Equation~\eqref{eq:eigenunitary} now contradicts \(\chi(t)\ne1\).
\end{proof}

\subsection{The rotation tensor-factor comparison}
Let \(Q=\Gamma_0\times\ent^2\) act on \(\M=\L(A)\) through \(\Gamma_0\). For \(\theta\in\real\setminus\rat\), define
\begin{equation}\label{eq:sigma}
\sigma_\theta((g,m,n),(h,p,q))=e^{2\pi i\theta np}1_\M.
\end{equation}
This is a scalar cocycle, and \(Q\) is torsion-free and residually finite.

\begin{proposition}\label{prop:scalar}
There is a trace-preserving isomorphism
\[
\M\rtimes_{\beta,\sigma_\theta}Q
\cong\L(E)\barot\L_{\sigma_\theta}(\ent^2).
\]
This crossed product is non-QWEP. Moreover \(\sigma_\theta\) is not a coboundary for this action with values in \(\cU(\M)\).
\end{proposition}
\begin{proof}
The two commuting families \(\M,\{\lambda_{\sigma_\theta}(g,0,0)\}\) and \(\{\lambda_{\sigma_\theta}(e,m,n)\}\) generate the respective tensor factors, and their canonical traces multiply. This gives the isomorphism. Applying \(\id\barot\tau_{\sigma_\theta}\) yields the non-QWEP expected subalgebra \(\L(E)\).

On \(\{e\}\times\ent^2\), the action is trivial and \(\M\) is abelian. Every \(\cU(\M)\)-valued coboundary \(b_sb_tb_{s+t}^*\) is therefore symmetric. In contrast,
\[
\sigma_\theta(e_2,e_1)=e^{2\pi i\theta}\ne1=\sigma_\theta(e_1,e_2).
\]
\end{proof}

\section{Scalar cocycles}\label{sec:twisted}
We now compare scalar-twisted crossed products with ordinary crossed products using the absorption principle in \cite[Proposition~2.6]{xxz26}. We first give the diagonal embeddings for a general von Neumann algebra.

\begin{proposition}\label{prop:diagonal}
Let $\alpha:G\curvearrowright \M$ be an action and let $\sigma\in Z^2(G,\T)$. Put $\R_1=\M\rtimes_\alpha G$ and $\R_\sigma=\M\rtimes_{\alpha,\sigma}G$. There are faithful normal embeddings with normal conditional expectations
\[
\R_\sigma\hookrightarrow \R_1\barot\L_\sigma(G),
\qquad
\R_1\hookrightarrow \R_\sigma\barot\L_{\bar\sigma}(G).
\]
Consequently, if $\L_\sigma(G)$ is QWEP, then $\R_\sigma$ is QWEP if and only if $\R_1$ is QWEP.
\end{proposition}
\begin{proof}
Write $\lambda(g)$ and $\lambda_\sigma(g)$ for the canonical unitaries. By \cite[Proposition~2.6]{xxz26}, the maps
\[
x\lambda_\sigma(g)\longmapsto x\lambda(g)\otimes\lambda_\sigma(g),\qquad
x\lambda(g)\longmapsto x\lambda_\sigma(g)\otimes\lambda_{\bar\sigma}(g)
\]
are $*$-homormophisms between von Neumann algebras. In the regular representations,
their ranges are the crossed products associated with the diagonal subgroup
$\Delta G\subset G\times G$. There exists a conditional expectation from $\R_1 \overline{\otimes} \L_\sigma(G)$ to the diagonal subalgebra. The conditional expectations satisfy
\[
\E_{\Delta G}(x\lambda(g)\otimes\lambda_\sigma(h))
=\delta_{g,h}\,x\lambda(g)\otimes\lambda_\sigma(h),
\]
and similarly for the second map. Since
$\L_{\bar\sigma}(G)\cong\L_\sigma(G)^{\rm op}$, tensor-product and
expectation permanence give the conclusion.
\end{proof}

\begin{theorem}[{Theorem \ref{thm:positiveoverview}}]\label{thm:scalarfactorization}
Let $\alpha:G\curvearrowright \M$ admit an invariant normal state $\varphi$, and let $\sigma\in Z^2(G,\T)$. Then
\begin{equation}\label{eq:scalarfactorization}
\M\rtimes_{\alpha,\sigma}G\text{ is QWEP}
\quad\Longleftrightarrow\quad
\M\rtimes_\alpha G\text{ and }\L_\sigma(G)\text{ are QWEP}.
\end{equation}
In particular, this holds for every finite trace-preserving action.
\end{theorem}
\begin{proof}
Invariance of $\varphi$ gives the normal conditional expectation
\[
\E_\varphi:\M\rtimes_{\alpha,\sigma}G\longrightarrow\L_\sigma(G),
\qquad \E_\varphi(x\lambda_\sigma(g))=\varphi(x)\lambda_\sigma(g).
\]
Indeed, apply $\varphi$ to the coefficient entries in the regular
representation. Thus QWEP of the crossed product implies QWEP of
$\L_\sigma(G)$. Both directions now follow from
Proposition~\ref{prop:diagonal}.
\end{proof}

\begin{corollary}\label{cor:allscalartwists}
If a finite trace-preserving action has a non-QWEP ordinary crossed product, all its scalar-twisted crossed products are non-QWEP. In particular, this holds for each action in Theorem~\ref{thm:overview}\textup{(i)}.
\end{corollary}
\begin{proof}
Apply Theorem~\ref{thm:scalarfactorization}.
\end{proof}

\begin{lemma}\label{lem:scalar}
If $\L_\sigma(G)$ is QWEP, then $\L(G)$ is QWEP. Consequently, QWEP of a finite trace-preserving scalar-twisted crossed product implies QWEP of both its coefficient algebra and $\L(G)$.
\end{lemma}
\begin{proof}
The diagonal map
\[
\lambda(g)\longmapsto\lambda_\sigma(g)\otimes\lambda_{\bar\sigma}(g)
\]
embeds $\L(G)$ trace-preservingly into $\L_\sigma(G)\barot\L_{\bar\sigma}(G)$. The two factors are QWEP simultaneously, and the image is the image of a conditional expectation. For the crossed-product assertion, use the conditional expectation from $\R$ to $\L_{\sigma}(G)$.
\end{proof}

\begin{corollary}\label{cor:CHactions}
For a hyperlinear group $G$, the following are equivalent:
\begin{enumerate}[label=\textup{(\roman*)},leftmargin=2em]
\item $G$ is cocycle-hyperlinear;
\item every finite trace-preserving action of $G$ with QWEP ordinary crossed product has QWEP scalar-twisted crossed products.
\end{enumerate}
\end{corollary}
\begin{proof}
The first implication follows from Theorem~\ref{thm:scalarfactorization}. For the converse, apply \textup{(ii)} to the trivial action on $\com$.
\end{proof}

\begin{remark}
For an arbitrary action, Proposition~\ref{prop:diagonal} still gives
\[
\M\rtimes_\alpha G\text{ QWEP},\quad
\M\rtimes_{\alpha,\sigma}G\text{ non-QWEP}
\quad\Longrightarrow\quad \L_\sigma(G)\text{ non-QWEP}.
\]
An invariant semifinite trace does not suffice for Theorem~\ref{thm:scalarfactorization}; see Proposition~\ref{prop:translation}.
\end{remark}

\section{Exterior equivalence and projective representations}\label{sec:typeI}
For a nonzero Hilbert space $K$, an action $\beta:G\curvearrowright B(K)$ need not lift to a unitary representation. Its scalar multiplier determines the crossed product.

\begin{proposition}%[Lifting and the crossed-product decomposition]
\label{prop:typeIdecomp}
Choose \(u(e)=1\) and \(u(g)\in\cU(K)\) with \(\beta_g=\Ad_{u(g)}\). There is a normalized scalar \(2\)-cocycle \(\sigma\) such that
\[
u(g)u(h)=\sigma(g,h)u(gh).
\]
Its cohomology class in \(H^2(G,\T)\) is independent of these choices. The action is implemented by a genuine unitary representation on \(K\) if and only if \([\sigma]=0\). Moreover,
\begin{equation}\label{eq:typeIdecomp}
B(K)\rtimes_\beta G\ \cong\ B(K)\barot\L_{\bar\sigma}(G).
\end{equation}
Consequently this crossed product is QWEP if and only if \(\L_{\bar\sigma}(G)\) is QWEP.
\end{proposition}
\begin{proof}
Since $\Ad_{u(g)u(h)u(gh)^*}=\id$, there is $\sigma(g,h)\in\T$ with $u(g)u(h)=\sigma(g,h)u(gh)$. 

Associativity gives the cocycle identity, Any other normalized choice of implementing unitaries has the form \(u'(g)=b_gu(g)\), where \(b_g\in\mathbb T\) and \(b_e=1\). The corresponding cocycle is
\[
\sigma'(g,h)=b_gb_h\overline{b_{gh}}\,\sigma(g,h),
\]
so \(\sigma'\) and \(\sigma\) differ by a scalar coboundary and define the same class in \(H^2(G,\mathbb T)\).
This proves the lifting criterion.

For the canonical crossed-product unitaries $\lambda(g)$, put $w(g)=u(g)^*\lambda(g)$.
Then
\[
[w(g),x]=0\quad(x\in B(K)),\qquad
w(g)w(h)=\overline{\sigma(g,h)}w(gh).
\]
In the regular representation on $\ell_2(G,K)$, let
$(D\xi)(t)=u(t^{-1})^*\xi(t)$. Then
\[
D\pi_\beta(x)D^*=1\otimes x,\qquad
D\lambda(g)D^*=\lambda_{\bar\sigma}(g)\otimes u(g).
\]
This proves \eqref{eq:typeIdecomp}. Since $B(K)$ is injective,
tensor permanence and a normal-state slice give the QWEP equivalence.
\end{proof}

\begin{remark}%[A finite action which does not lift]
\label{rem:pauli}

In this remark we construct a finite action which can not be lifted.

Let \(G=(\ent/2\ent)^2\), \(K=\com^2\), and
\[
U=\begin{pmatrix}1&0\\0&-1\end{pmatrix},\qquad
V=\begin{pmatrix}0&1\\1&0\end{pmatrix},\qquad UV=-VU.
\]
The action $\beta_{(a,b)}=\Ad_{U^aV^b}$ has multiplier
\[
\sigma((a,b),(c,d))=(-1)^{bc}.
\]
Every implementing pair is a scalar multiple of $U,V$, so $\beta$ does not lift. The four monomials $U^aV^b$ form a basis of $\bM_2$; hence
\[
\bM_2\rtimes_\beta G\cong \bM_4,
\qquad
\bM_2\barot\L(G)\cong\bigoplus_{j=1}^{4}\bM_2.
\]
\end{remark}

\subsection{Projective matrix approximations}

\begin{proposition}%[Projective approximation criterion]
\label{prop:projectivecriterion}
Assume \(G\) is hyperlinear, and retain the multiplier \(\sigma\) from Proposition~\ref{prop:typeIdecomp}. The following are equivalent:
\begin{enumerate}[label=\textup{(\roman*)},leftmargin=2em]
\item \(B(K)\rtimes_\beta G\) is QWEP.
\item There are dimensions \(d_n\) and maps \(r_n:G\to\cU(\bM_{d_n})\), with \(r_n(e)=1\), such that for every \(g,h\in G\),
\begin{equation}\label{eq:projectiveapprox}
\lim_{n\to\infty}
\bigl\|r_n(g)r_n(h)-\sigma(g,h)r_n(gh)\bigr\|_{2,d_n}=0.
\end{equation}
\end{enumerate}
No trace-separation condition is imposed on the maps \(r_n\).
\end{proposition}
\begin{proof}
If \textup{(i)} holds, $\L_\sigma(G)$ is Connes embeddable.
Unitary microstates and a diagonal choice over $G$ give \textup{(ii)}.

Conversely, put
\[
r(g)=[r_n(g)]_\omega,\qquad \Pcal=\{r(g):g\in G\}'',\qquad
W_g=\lambda(g)\otimes\overline{r(g)}\in\L(G)\barot\overline \Pcal.
\]
Both tensor factors are QWEP, and
\[
W_gW_h=\bar\sigma(g,h)W_{gh},\qquad
(\tau_G\otimes\bar\tau_\Pcal)(W_g)=\delta_{g,e}.
\]
Hence $\{W_g:g\in G\}''\cong\L_{\bar\sigma}(G)$ trace-preservingly.
The trace-preserving expectation and Proposition~\ref{prop:typeIdecomp}
give \textup{(i)}.
\end{proof}

\begin{corollary}%[Sufficient conditions]
\label{cor:typeIpositive}
For hyperlinear \(G\), the crossed product \(B(K)\rtimes_\beta G\) is QWEP in each of the following cases:
\begin{enumerate}[label=\textup{(\roman*)},leftmargin=2em]
\item \([\sigma]=0\), equivalently \(\beta\) has a genuine implementing representation;
\item \(\sigma\) has a finite-dimensional projective unitary realization, possibly on a space different from \(K\);
\item \(\beta\) admits an invariant state on \(B(K)\), not necessarily a normal state.
\end{enumerate}
In particular, the answer is positive whenever \(K\) is finite-dimensional.
\end{corollary}
\begin{proof}
Cases \textup{(i)} and \textup{(ii)} follow from
Proposition~\ref{prop:projectivecriterion} with exact projective models.
For \textup{(iii)}, put $\mathcal D=C^*(u(g):g\in G)$. Invariance gives
\[
\varphi(u(g)x)=\varphi(xu(g))\qquad(x\in B(K),\ g\in G).
\]
Thus $\varphi|_\mathcal D$ is an amenable trace. By
\cite[Theorem~6.1, \textup{(i)}$\Rightarrow$\textup{(ii')}]{Ozawa04},
there is a unital $*$-homomorphism $\theta:\mathcal D\to \R_0^\omega$.
Since
\[
\theta(u(g))\theta(u(h))=\sigma(g,h)\theta(u(gh)),
\]
the tensor argument in Proposition~\ref{prop:projectivecriterion} applies.
\end{proof}

For the mixed cocycle of Section~\ref{sec:mixed}, the finite quotient is explicit:
\[
q:\Gamma_0\times\ent\longrightarrow\mu_3\times(\ent/3\ent),
\qquad q(g,m)=(\chi(g),m\bmod3),
\]
where \(\mu_3\) is the group of cube roots of unity. The cocycle is the pullback of
\(\omega((a,\bar m),(b,\bar n))=a^n\).
Consequently \(\L_\sigma(\Gamma_0\times\ent)\) is QWEP by the same finite-dimensional tensor argument. This proves the QWEP assertion for the twisted group algebra in Theorem~\ref{thm:overview}\textup{(ii)}.

\subsection{Relation to central extensions}
Every scalar cocycle is realized by \(\Ad_{\lambda_\sigma(\cdot)}\) on \(B(\ell_2(G))\). Therefore a positive answer for all actions on type-I factors is equivalent to the assertion that every scalar-twisted group algebra of a hyperlinear group is QWEP.

More precisely, let \(A_\sigma\leq\T\) be the countable subgroup generated by the values of \(\sigma\), and give \(E_\sigma=A_\sigma\times G\) the product
\[
(a,g)(b,h)=(ab\sigma(g,h),gh).
\]
Then
\begin{equation}\label{eq:centralextensioncriterion}
\L_\sigma(G)\text{ is QWEP}
\quad\Longleftrightarrow\quad
E_\sigma\text{ is hyperlinear}.
\end{equation}
For the forward direction, embed \(\L_\sigma(G)\) into \(\R_0^\omega\); the map \((a,g)\mapsto a\lambda_\sigma(g)\) is an injective group homomorphism. For the reverse direction, use the central-extension characterization in \cite[Lemma 4.8]{DV26}, with the character \(A_\sigma\hookrightarrow\T\). That characterization also shows that the assertion for all hyperlinear \(G\) and all scalar \(\sigma\) is equivalent to closure of countable hyperlinear groups under central extensions. 

\subsection{Positive classes from cocycle-hyperlinearity}\label{sec:CFpositive}
\begin{corollary}\label{cor:CFtypeI}
If \(G\) is cocycle-hyperlinear, then \(B(K)\rtimes_\beta G\) is QWEP for every action \(\beta\) and every nonzero Hilbert space \(K\). This applies to amenable-by-treeable groups, to
\[
\SL_n(\ent),\quad \operatorname{GL}_n(\ent),\quad
\operatorname{PSL}_n(\ent)\qquad(n\geq2),
\]
and to their countable central extensions.
\end{corollary}
\begin{proof}
Combine Proposition~\ref{prop:typeIdecomp} with
\cite[Theorem~A, Theorem~4.4 and Corollary~4.5]{CF26}.
\end{proof}

\subsection{Exact projective representations}
\begin{proposition}\label{prop:laurent-inflation}
Let $R=\FF_{11}[t,t^{-1}]$ and $Q=\SL_3(R)$. Then $Q$ is residually finite and $H^2(Q,\T)\cong C_{10}$. If $[\sigma]\ne0$, then the full twisted group $C^*$-algebra $C^*_\sigma(Q)$ has no nonzero finite-dimensional representation, and $[\sigma]$ is not inflated from a finite quotient.
\end{proposition}
\begin{proof}
By \cite[Theorem~5]{Sug21} and \cite[Lemma~6.4]{KLPV25}, the universal central extension is
\[
1\longrightarrow C_{10}\longrightarrow
U=\operatorname{St}_3(R)\longrightarrow Q\longrightarrow1.
\]
Thus $H^2(Q,\T)\cong\widehat{C_{10}}$. Finite-ring reductions give residual finiteness.

For $\pi:U\to\cU(\bM_n)$, the root-subgroup images are finite abelian of exponent $11$. Weyl conjugation and the Steinberg relations make their common kernel a finite-index ideal $I\lhd R$. Hence $\pi$ factors through $\operatorname{St}_3(B)$, $B=R/I$. By \cite[Lemma~17 and its proof]{KS09}, $K_{2,3}(B)=K_2(B)$. In each local factor, unit symbols generate $K_2$ and are bimultiplicative \cite[\S3, Theorem~3.2]{KN06}. Coefficient-field symbols vanish \cite[Lemma~6.1\textup{(1)}]{KLPV25}, and principal units have $11$-power order. Thus $K_2(B)$ is $11$-primary, so $\pi$ kills $C_{10}$.

A finite-dimensional $\sigma$-projective representation lifts to $U$ with central character $[\sigma]$, which must be trivial. An inflated class would admit such a representation by the twisted regular representation of the finite quotient.
\end{proof}

\begin{example}\label{ex:fd-gap-qwep}
Let
\[
P=\ent[1/5]^2\rtimes\ent,\qquad
(a,b,k)(a',b',k')=(a+a'+kb',b+b',k+k').
\]
By \cite[Definition~2.2 and Lemma~2.3]{Eckhardt25}, $P$ is amenable and residually finite. Thus $B=C^*(P)$ is nuclear and residually finite-dimensional, since its finite-quotient traces converge to the faithful canonical trace.

Put $b_0=(0,1,0)$, $g_0=(1/5,0,0)$ and $\zeta=e^{2\pi i/5}$. The normal closure of $b_0$ is $H=\ent^2$, since $[(0,0,1),b_0]=(1,0,0)$. In $E=P/H$, $z=g_0H$ is central of order $5$. For $G=E/\langle z\rangle$, the character $z\mapsto\zeta$ defines $\sigma$ with
\[
C^*_\sigma(G)\cong B/I,\qquad
I=\overline{\langle u_{b_0}-1,u_{g_0}-\zeta1\rangle}.
\]
Here $I$ is the closed two-sided ideal, and $\L_\sigma(G)$ is injective since $G$ is amenable.

In any irreducible finite-dimensional representation $\rho$ of $P$, the first coordinate acts by a character $\chi$ of finite order prime to $5$ \cite[Lemmas~2.1 and~2.7]{Eckhardt25}. If $\chi(1)=1$, then $\rho(g_0)=I$; otherwise $\tr(\rho(b_0))=0$. Since $|1-\zeta|^2<2$, every unital $\pi:B\to\bM_n$ satisfies
\begin{equation}\label{eq:fd-gap-injective}
\|\pi(u_{b_0})-I_n\|_{2,n}^2+
\|\pi(u_{g_0})-\zeta I_n\|_{2,n}^2
\geq |1-\zeta|^2>0,
\end{equation}
where $\|x\|_{2,n}^2=\tr_n(x^*x)$. Hence $C^*_\sigma(G)$ has no nonzero finite-dimensional representation and $[\sigma]\ne0$.
\end{example}

\section{Scalar twists of regular wreath products}\label{sec:regularwreath}
We use finite quotients and the bicharacters describing commutation between different coordinates; compare \cite[Lemma~5.2]{CF26}.

\begin{theorem}[{Theorem \ref{thm:wreathoverview}}]\label{thm:twistedregularwreath}
Let \(A,\Gamma\) be countable groups, with \(\Gamma\) residually finite, and put \(W=A^{(\Gamma)}\rtimes\Gamma=A\wr\Gamma\). Identify \(A\) with its identity-coordinate copy \(A_e\). For a normalized scalar cocycle \(\sigma\) on \(W\), put \(\alpha=\sigma|_{A_e\times A_e}\) and \(\eta=\sigma|_{\Gamma\times\Gamma}\). Then
\begin{equation}\label{eq:wreathcriterion}
\L_\sigma(W)\text{ is QWEP}
\quad\Longleftrightarrow\quad
\L_\alpha(A)\text{ and }\L_\eta(\Gamma)\text{ are QWEP}.
\end{equation}
\end{theorem}
\begin{proof}
Necessity follows from the expected subgroup inclusions. Conversely, let \(\pi:W\to\Gamma\) be the quotient, and put
\[
\rho(w,z)=\sigma(w,z)\overline{\eta(\pi(w),\pi(z))}.
\]
Then \(\rho|_\Gamma=1\) and \(\rho|_{A_e}=\alpha\). In \(\L_\rho(W)\), let \(V(g)\) and \(U_e(a)\) be the canonical unitaries associated with \(\Gamma\) and \(A_e\), respectively. Put \(U_s(a)=V(s)U_e(a)V(s)^*\). We have
\begin{align}
U_s(a)U_s(b)&=\alpha(a,b)U_s(ab),\label{eq:lampwithin}\\
V(g)U_s(a)V(g)^*&=U_{gs}(a).\label{eq:lampcovariance}
\end{align}
For \(s\ne t\), the scalar commutator of \(U_s(a),U_t(b)\) is multiplicative in each variable. Thus it factors through \(B=A_{\mathrm{ab}}\); covariance gives bicharacters \(\kappa_r:B\times B\to\T\), \(r\ne e\), with
\begin{align}
U_s(a)U_t(b)&=\kappa_{s^{-1}t}(\bar a,\bar b)U_t(b)U_s(a),\label{eq:lampcross}\\
\kappa_{r^{-1}}(y,x)&=\overline{\kappa_r(x,y)}.\label{eq:lampsymmetry}
\end{align}

Fix finitely many words, including adjoints. Move their group letters to the right using \eqref{eq:lampcovariance}, and let \(S,T\subset\Gamma\) be the resulting coordinate indices and terminal group labels. Choose a finite quotient \(q:\Gamma\to Q\) injective on \(\{e\}\cup S\cup S^{-1}S\), with \(q(t)\ne e\) for \(t\in T\setminus\{e\}\). Define bicharacters on \(Q\setminus\{e\}\) by
\[
\kappa^Q_{q(s^{-1}t)}=\kappa_{s^{-1}t}\quad(s,t\in S,\ s\ne t),
\qquad
\kappa^Q_{r^{-1}}(y,x)=\overline{\kappa^Q_r(x,y)},
\]
and use the trivial bicharacter on the remaining inverse pairs. This is consistent: \(q\) separates the inverse-closed difference set, and actual involutions already satisfy \eqref{eq:lampsymmetry}.

Order \(Q=\{q_1,\ldots,q_m\}\), and define
\[
\chi_Q(x,y)=\prod_{i>j}\kappa^Q_{q_i^{-1}q_j}(x_i,y_j)
\qquad(x,y\in B^Q),
\]
\[
c_Q(a,b)=\left(\prod_{i=1}^m\alpha(a_i,b_i)\right)\chi_Q(\bar a,\bar b),
\qquad \M_Q=\L_{c_Q}(A^Q).
\]
The bicharacter \(\chi_Q\) is a cocycle. The coordinate generators of \(\M_Q\) satisfy \eqref{eq:lampwithin} and \eqref{eq:lampcross} with \(\kappa^Q\): the factors in \(\chi_Q\) arise from ordering the product of two coordinate monomials. Moreover,
\begin{equation}\label{eq:finite-lamp-embedding}
\M_Q\longrightarrow
\left(\barot_{i=1}^m\L_\alpha(A)\right)\barot\L_{\chi_Q}(B^Q)
\end{equation}
is the trace-preserving embedding given by
\[
\lambda_{c_Q}(a)\longmapsto
\left(\bigotimes_{i=1}^m\lambda_\alpha(a_i)\right)\otimes\lambda_{\chi_Q}(\bar a).
\]
Its first factor supplies trace zero for \(a\ne e\). Since \(B^Q\) is abelian, \(\L_{\chi_Q}(B^Q)\) is injective; hence \(\M_Q\) is Connes embeddable.

The assignments \(U_s(a)\mapsto U_{rs}(a)\), \(r,s\in Q\), preserve the coordinate relations, because
\[
(rs)^{-1}(rt)=s^{-1}t.
\]
Ordered monomials form a basis indexed by \(A^Q\), so these assignments define a genuine trace-preserving action of \(Q\) on \(\M_Q\). Its crossed product satisfies
\[
\Pcal_Q=\M_Q\rtimes Q\hookrightarrow \bM_{|Q|}\barot \M_Q
\]
trace-preservingly, and is therefore Connes embeddable.

Replace \(U_s(a),V(g)\) in the chosen words by the corresponding \(q(s)\)-coordinate and \(q(g)\)-unitaries in \(\Pcal_Q\). Their traces agree: nonidentity terminal labels and nontrivial coordinate configurations remain nontrivial, while words reducing to the identity retain all \(\alpha\)-, \(\kappa\)- and inverse phases. Approximating these finite collections of moments by matrices proves that \(\L_\rho(W)\) is Connes embeddable. Finally,
\[
\L_\sigma(W)\hookrightarrow\L_\rho(W)\barot\L_\eta(\Gamma),
\qquad
\lambda_\sigma(w)\longmapsto\lambda_\rho(w)\otimes\lambda_\eta(\pi(w))
\]
is trace-preserving, since \(\rho\,\pi^*\eta=\sigma\). This proves sufficiency.
\end{proof}

\begin{corollary}\label{cor:CHwreath}
For countable \(A\) and residually finite \(\Gamma\),
\[
A\wr\Gamma\text{ is cocycle-hyperlinear}\quad\Longleftrightarrow\quad
A\text{ and }\Gamma\text{ are cocycle-hyperlinear}.
\]
For amenable \(A\), criterion \eqref{eq:wreathcriterion} reduces to the restriction to \(\Gamma\). In particular, \(\ent\wr\Gamma\) is cocycle-hyperlinear exactly when \(\Gamma\) is cocycle-hyperlinear.
\end{corollary}
\begin{proof}
Sufficiency follows from Theorem~\ref{thm:twistedregularwreath}. Conversely, every cocycle on \(\Gamma\) inflates to \(W\). For \(\alpha\in Z^2(A,\T)\), put
\[
c(a,b)=\prod_{s\in\Gamma}\alpha(a_s,b_s),\qquad
\widetilde c((a,g),(b,h))=c(a,g\cdot b).
\]
The finite product \(c\) is shift-invariant, so \(\widetilde c\) is a cocycle on \(W\) restricting to \(\alpha\) on \(A_e\). Expected subgroup inclusions give necessity. If \(A\) is amenable, all its scalar-twisted group algebras are injective.
\end{proof}

Together with \cite[Theorem B]{CF26}, this gives cocycle-hyperlinearity of
\[
\ent\wr\SL_n(\ent),\quad\ent\wr\operatorname{GL}_n(\ent),\quad
\ent\wr\operatorname{PSL}_n(\ent)\qquad(n\geq2),
\]
as well as \(\SL_3(\ent)\wr \FF_2\) and \(\SL_3(\ent)\wr\SL_3(\ent)\).

\begin{remark}\label{rem:augmentation}
Let \(W=\ent^{(\Gamma)}\rtimes\Gamma\), \(I_\Gamma=\ker(\sum:\ent^{(\Gamma)}\to\ent)\), and \(H=I_\Gamma\rtimes\Gamma\). For \(\rho\in Z^2(\Gamma,\T)\), put
\[
\Omega_\rho((a,g),(b,h))=\prod_{t\in\Gamma}\rho(g,t)^{b_t}.
\]
The cocycle identity follows from that of $\rho$ and $\sum_t a_t=0$. Its restrictions along
\[
g\longmapsto(0,g),\qquad g\longmapsto(\delta_g-\delta_e,g)
\]
are $1$ and $\rho$. These embeddings are conjugate by $(-\delta_e,e)\in W$, so extension of $[\Omega_\rho]$ to $W$ forces $[\rho]=0$. Conversely, if $\rho=\delta c$, then $\Omega_\rho=\delta F$ for $F(a,g)=\prod_t c(t)^{a_t}$. Thus $[\Omega_\rho]$ extends to $W$ exactly when $[\rho]=0$.
\end{remark}

\section{Fourier multipliers and Schur multipliers}\label{sec:fourier}
In this section, for complementing the point of view of QWEP stability, we start to discuss the multiplier transference results obtained by Xiong, Xu and the author. All related results are found and cited from \cite{xxz26} and \cite{Zeng23}.

Let $(\M,G,\alpha,\sigma)$ be a trace-preserving twisted dynamical system.
Amenability of $\alpha$ is understood in the sense of
\cite[Definition~2.15]{xxz26}.
As in \cite[Theorems~3.2--3.3 and~4.1--4.2]{xxz26}, put
\[
\R=\M\rtimes_{\alpha,\sigma}G,\qquad
\N=\M\barot B(\ell_2(G)).
\]
We identify $\M$ with $\pi_\alpha(\M)$. For
$\varphi\in\ell_\infty(G)$, define
\[
T_\varphi\Bigl(\sum_{s\in F}a_s\lambda_\sigma(s)\Bigr)
=\sum_{s\in F}\varphi(s)a_s\lambda_\sigma(s),
\qquad a_s\in\M,
\]
where $F\subset G$ is finite. We use the notation
$\mathsf M(L_p(\R))$ and $\mathsf M_{\cb}(L_p(\R))$ for the spaces
of bounded and completely bounded Fourier multiplier symbols, with norms
\[
\|\varphi\|_{\mathsf M(L_p(\R))}=\|T_\varphi\|,
\qquad
\|\varphi\|_{\mathsf M_{\cb}(L_p(\R))}=\|T_\varphi\|_{\cb}.
\]
For $\widetilde\varphi(s,t)=\varphi(st^{-1})$, the Schur multiplier is
\[
M_{\widetilde\varphi}([a_{s,t}])
=[\varphi(st^{-1})a_{s,t}].
\]
The corresponding symbol spaces are denoted by $\mathsf S$ and
$\mathsf S_{\cb}$, on either $S_p(\ell_2(G))$ or $L_p(\N)$.
We put the multiplier norm equal to $\infty$ when no bounded extension exists. 

From now on, all the results concerning transference of multipliers are recorded from \cite{xxz26} and \cite{Zeng23} only for completeness and discussions, no novelties and originalities are claimed in the results.

\begin{proposition}[{\cite[Theorem~3.2]{xxz26}}]\label{prop:fourier-qwep}
Let $1<p<\infty$. If $\R$ is QWEP, then
\[
\mathsf M_{\cb}(L_p(\L(G)))
\subset\mathsf M_{\cb}(L_p(\R))\quad\text{contractively}.
\]
\end{proposition}
\begin{proof}
The trace-preserving embedding
\[
\Phi:\R\longrightarrow\R\barot\L(G),\qquad
\Phi(a\lambda_\sigma(s))=a\lambda_\sigma(s)\otimes\lambda(s)
\]
induces a complete isometry on $L_p$. Write $\widetilde T_\varphi$
for the Fourier multiplier on $L_p(\L(G))$. Then
\[
\Phi T_\varphi=(\id\otimes\widetilde T_\varphi)\Phi,
\qquad
\|T_\varphi\|_{\cb}
\le\|\id\otimes\widetilde T_\varphi\|_{\cb}
=\|\widetilde T_\varphi\|_{\cb}.
\]
The last equality uses the characterization of QWEP von Neumann algebras (see \cite{Junge04}); its ultraproduct
construction is given in \cite[Theorem~3.11 and Corollary~3.12]{Junge04}.
\end{proof}

\begin{corollary}[{\cite[Theorem~3.3]{xxz26}}]\label{cor:scalar-fourier-qwep}
Let $\sigma\in Z^2(G,\T)$ and $1<p<\infty$. If $\L_\sigma(G)$ is QWEP, then
\[
\mathsf M_{\cb}(L_p(\L_\sigma(G)))
=\mathsf M_{\cb}(L_p(\L(G)))\quad\text{isometrically}.
\]
\end{corollary}
\begin{proof}
One inclusion follows from Proposition~\ref{prop:fourier-qwep}.
For the converse, $\L_{\bar\sigma}(G)$ is QWEP and
\[
\lambda(s)\longmapsto
\lambda_{\bar\sigma}(s)\otimes\lambda_\sigma(s)
\]
is a trace-preserving embedding. Apply the same tensor-extension
argument to $\id\otimes T_\varphi$.
\end{proof}

\begin{lemma}\label{lem:amenable-supported-symbols}
Let $H\leq G$ be amenable, $\sigma\in Z^2(G,\T)$ and $1<p<\infty$.
If $\varphi\in\ell_\infty(G)$ is supported on $H$, then
\begin{equation}\label{eq:amenable-supported-symbols}
\|\varphi\|_{\mathsf M_{\cb}(L_p(\L_\sigma(G)))}
=\|\varphi\|_{\mathsf M_{\cb}(L_p(\L(G)))}.
\end{equation}
\end{lemma}
\begin{proof}
Subgroup inclusions and expectations reduce both norms to $H$. Apply Corollary~\ref{cor:scalar-fourier-qwep}, since $\L_{\sigma|_H}(H)$ is injective.
\end{proof}

\begin{corollary}\label{cor:multiplier-obstruction}
If, for some $1<p<\infty$ and $\varphi\in\ell_\infty(G)$,
\[
\|\varphi\|_{\mathsf M_{\cb}(L_p(\L(G)))}
<\|\varphi\|_{\mathsf M_{\cb}(L_p(\R))},\qquad
\|\varphi\|_{\mathsf M_{\cb}(L_p(\L(G)))}<\infty,
\]
then $\R$ is not QWEP.
\end{corollary}
\begin{proof}
Apply Proposition~\ref{prop:fourier-qwep}.
\end{proof}

The next comparison extends \cite[Lemma~2.4 and Theorem~2.7\textup{(b)}]{NR11} and
\cite[Theorem~A]{Gonz} to twisted crossed products.

\begin{proposition}[{\cite[Theorem~4.1\textup{(ii)}, Theorem~4.2 and Corollary~4.3]{xxz26}}]\label{prop:fourier-schur}
Let $1<p<\infty$. Then
\[
\|\widetilde\varphi\|_{\mathsf S_{\cb}(S_p(\ell_2(G)))}
\le\|\widetilde\varphi\|_{\mathsf S_{\cb}(L_p(\N))}
\le\|\varphi\|_{\mathsf M_{\cb}(L_p(\R))}.
\]
If $\alpha$ is amenable, the second inequality is an equality.
If additionally $\M$ is QWEP, all three norms are equal.
\end{proposition}
\begin{proof}
%The first inequality follows by restriction. For finite $F\subset G$, put
%\[
%D_F=\sum_{s\in F}\lambda_\sigma(s)\otimes e_{s,s},\qquad
%a=\sum_{s,t\in F}a_{s,t}\otimes e_{s,t},\quad a_{s,t}\in\M.
%\]
%Since %$(D_FaD_F^*)_{s,t}\in\M\lambda_\sigma(st^{-1})$%,
%\[
%(T_\varphi\otimes\id)(D_FaD_F^*)
%=D_F(\id\otimes M_{\widetilde\varphi})(a)D_F^*.
%\]
%Taking $L_p$-norms and matrix amplifications proves the second inequality.

%If $\alpha$ is amenable, the completely contractive maps in
%\cite[Lemmas~3.1.14\textup{(ii)} and~3.1.16]{Zeng23} satisfy
%\[
%\Phi_{p,i}:L_p(\R)\longrightarrow L_p(\N),\qquad
%\Psi_{p,i}:L_p(\N)\longrightarrow L_p(\R),
%\]
%\[
%\Psi_{p,i}\Phi_{p,i}(x)\longrightarrow x,\qquad
%\Phi_{p,i}T_\varphi=(\id\otimes M_{\widetilde\varphi})\Phi_{p,i}.
%\]
%Hence $\|T_\varphi\|_{\cb}\le
%\|\id\otimes M_{\widetilde\varphi}\|_{\cb}$.
%If $\M$ is QWEP, tensor extension gives
%$\|\id\otimes M_{\widetilde\varphi}\|_{\cb}
%=\|M_{\widetilde\varphi}\|_{\cb}$.
See corresponding results in \cite{xxz26} and \cite{Zeng23}.
\end{proof}

The cited theorems are stated for $2<p<\infty$.
Duality gives $1<p<2$, and $p=2$ follows from the Fourier decomposition.
Thus QWEP of $\R$ gives group-to-crossed-product transference, whereas
QWEP of $\M$ gives Schur-to-Fourier transference under amenability of $\alpha$.
The examples of Section~\ref{sec:ordinary} show that QWEP of $\M$ and
$\L(G)$ does not imply the first hypothesis. On the other hand,
Theorem~\ref{thm:twistedregularwreath} and
Corollary~\ref{cor:scalar-fourier-qwep} give cocycle-independent multiplier
spaces for $A\wr\Gamma$ whenever the two restricted twisted algebras are QWEP.

\begin{problem}\label{separation}
Find a hyperlinear group $G$, $\sigma\in Z^2(G,\T)$,
$1<p<\infty$, $p\ne2$, and $\varphi\in\ell_\infty(G)$ such that
\[
\|\varphi\|_{\mathsf M_{\cb}(L_p(\L(G)))}
<\|\varphi\|_{\mathsf M_{\cb}(L_p(\L_\sigma(G)))},\qquad
\|\varphi\|_{\mathsf M_{\cb}(L_p(\L(G)))}<\infty.
\]
\end{problem}
Such a symbol gives a non-QWEP twist, and its support generates a nonamenable subgroup. Example~\ref{ex:fd-gap-qwep} shows that an exact representation gap alone gives no such separation.

\subsection{An obstruction to problem \ref{separation}}

\begin{proposition}%[Norm of the coefficient identification]
\label{prop:coefficientnorm}
Let $n\geq2$, $\zeta=e^{2\pi i/n}$, $G_n=(\ent/n\ent)^2$, and put $Ue_k=\zeta^ke_k$, $Ve_k=e_{k+1}$ on $\com^n$. The cocycle $\sigma_n((a,b),(c,d))=\zeta^{-bc}$ identifies $\L_{\sigma_n}(G_n)$ with $(\bM_n,\tr_n)$. With uniform probability measure on $G_n$, the coefficient map
\[
T_n(U^aV^b)(j,k)=\zeta^{aj+bk}
\]
satisfies, for $1\leq p\leq\infty$,
\[
\|T_n\|_{L_p(\bM_n)\to L_p(G_n)}
=\|T_n^{-1}\|_{L_p(G_n)\to L_p(\bM_n)}
=n^{|1/2-1/p|}.
\]
\end{proposition}
\begin{proof}
For the unitary Fourier matrix $F_{lk}=n^{-1/2}\zeta^{-lk}$,
\[
(T_nx)(j,k)=\sqrt n\,\zeta^{jk}(xF)_{jk},\qquad
(T_n^{-1}f)F=n^{-1/2}[\zeta^{-jk}f(j,k)]_{j,k}.
\]
Both maps are $L_2$-isometries and have $L_\infty$-norm at most $\sqrt n$. Interpolation gives the upper bounds for $p\geq2$.
For $xF=E_{00}$ and $f(j,k)=\zeta^{jk}$,
\[
\frac{\|T_nx\|_p}{\|x\|_p}
=\frac{\|T_n^{-1}f\|_p}{\|f\|_p}=n^{1/2-1/p}.
\]
Since $T_n^*=T_n^{-1}$, duality gives $p\leq2$.
\end{proof}
We see that 
\[
\|T_n\|_{L_p(\bM_n)\to L_p(G_n)}
=\|T_n^{-1}\|_{L_p(G_n)\to L_p(\bM_n)}
=n^{|1/2-1/p|} \to \infty.
\]
However, for
\[
G=\bigoplus_{n\geq2}G_n,\qquad \sigma=\prod_{n\geq2}\sigma_n,
\]
$\L(G)$ is diffuse abelian and $\L_\sigma(G)\cong\barot_{n\geq2}(\bM_n,\tr_n)$ is hyperfinite, thus QWEP. Their completely bounded Fourier multiplier norms agree, but both coefficient identifications are unbounded on $L_p$ for $p\ne2$ by Proposition~\ref{prop:coefficientnorm}. Thus distortion of the Fourier multiplier norms do not provide the example we seek for problem \ref{separation}.

\section{Operator-valued cocycles}\label{sec:reverse}
\subsection{A free presentation}
The following elementary extension construction applies to \emph{any} finitely generated nonhyperlinear group, and in particular to the explicit torsion-free group \(E\).

\begin{theorem}[{Theorem \ref{thm:overview}, (iv)}]
\label{thm:reverse}
Let \(G\) be finitely generated and nonhyperlinear. Choose a surjection
\[
p:\FF_d\twoheadrightarrow G,\qquad K=\ker p,
\]
from a free group of finite rank, and a set-theoretic section \(s:G\to \FF_d\) with \(s(e)=e\). On \(\Pcal=\L(K)\), define
\begin{align}\label{eq:sectioncocycle}
\alpha_g(\lambda(k))&=\lambda(s(g)ks(g)^{-1}),\\
\sigma(g,h)&=\lambda(s(g)s(h)s(gh)^{-1}).\nonumber
\end{align}
Then \((\alpha,\sigma)\) is a trace-preserving cocycle action and
\begin{equation}\label{eq:freecrossed}
\Pcal\rtimes_{\alpha,\sigma}G\cong\L(\FF_d).
\end{equation}
Both \(\Pcal\) and this crossed product are QWEP, while \(\L(G)\) is not. No perturbation by unitaries in \(\Pcal\) can make the cocycle scalar-valued.
\end{theorem}
\begin{proof}
Put $c(g,h)=s(g)s(h)s(gh)^{-1}\in K$. Then
\[
c(g,h)c(gh,k)=s(g)c(h,k)s(g)^{-1}c(g,hk),
\]
which gives the cocycle identities in \eqref{eq:sectioncocycle}.
The map
\[
\sum_g x_g\lambda_\sigma(g)\longmapsto\sum_g x_g\lambda(s(g))
\]
is multiplicative and trace-preserving. Since
$\FF_d=\bigsqcup_{g\in G}Ks(g)$, its image generates $\L(\FF_d)$.
It therefore extends to \eqref{eq:freecrossed}.

The groups $K$ and $\FF_d$ are free, hence residually finite. Thus $\Pcal$
and $\L(\FF_d)$ are QWEP, whereas $\L(G)$ is not.
If $b_g\in\cU(\Pcal)$ made the new cocycle
\[
\sigma'(g,h)=b_g\alpha_g(b_h)\sigma(g,h)b_{gh}^*
\]
scalar-valued, the new action $\Ad_{b_g}\alpha_g$ would still preserve
the trace. Lemma~\ref{lem:scalar} would imply QWEP of $\L(G)$,
a contradiction.
\end{proof}

Taking $G=E$ proves Theorem~\ref{thm:overview}\textup{(iv)}.

\subsection{The translation action}\label{sec:translation}
Let \(G=E\) and let \(\beta_gf(h)=f(g^{-1}h)\) on \(\ell_\infty(G)\).

\begin{proposition}\label{prop:translation}
For the translation action,
\[
\ell_\infty(G)\rtimes_\beta G\cong B(\ell_2(G)).
\]
Hence this ordinary crossed product is QWEP while \(\L(G)\) is non-QWEP.
\end{proposition}
\begin{proof}
For $p_s=1_{\{s\}}$, the elements
\[
e_{s,t}=p_s\lambda(st^{-1})p_t
\]
are matrix units with $\sum_s e_{s,s}=1$ strongly and
$p_e(\ell_\infty(G)\rtimes G)p_e=\com p_e$.
Hence the crossed product is $B(\ell_2(G))$, which is injective.
\end{proof}

The action preserves the semifinite counting trace. Since \(G\) is nonamenable, it admits no invariant state on \(\ell_\infty(G)\). The inclusion \(\L(G)\subset B(\ell_2(G))\) is not expected, by QWEP permanence.

\subsection{A measured-groupoid construction}\label{sec:groupoid}
The same extension construction applies to measured groupoids. We use a probability-measure-preserving equivalence relation with non-QWEP von Neumann algebra from \cite[Theorem~1.1]{Manzoor26}.

\subsubsection{The relation algebra}
By \cite[Theorem 1.2]{Manzoor26}, choose a non-co-hyperlinear invariant random subgroup \(\nu\) of a finite-rank nonabelian free group \(\FF_r\). Thus the GNS algebra of the character
\[
\chi_\nu(g)=\nu\{H:g\in H\}
\]
is not Connes embeddable by \cite[Proposition~2.17\textup{(1)}]{Manzoor26}. Realize \(\nu\) as the stabilizer distribution of a probability-measure-preserving action \(\FF_r\curvearrowright(X,\mu)\), using \cite[Theorem 2.4]{Manzoor26}. Let \(\mathcal R\) be its orbit equivalence relation and set
\[
\A_X=L_\infty(X,\mu),\qquad \Pcal_X=\A_X\rtimes \FF_r.
\]
The character algebra embeds trace-preservingly into \(\L(\mathcal R)\), by \cite[Proposition 2.11 and Definition 2.12]{Manzoor26}. Hence \(\L(\mathcal R)\) is not Connes embeddable, and is non-QWEP.

\begin{lemma}\label{lem:freeaction}
For every probability-measure-preserving action of a finite-rank free group on a standard probability space, \(L_\infty(X)\rtimes \FF_r\) is QWEP. Freeness of the action is not required.
\end{lemma}
\begin{proof}
For a free basis $a_1,\ldots,a_r$,
\[
\Pcal_X\cong *_{\A_X,\,i=1}^{r}\bigl(\A_X\rtimes\langle a_i\rangle\bigr).
\]
Indeed, alternating centered Fourier polynomials have no identity
coefficient. Each factor is an amenable crossed product of an abelian
algebra and is QWEP by \cite[Proposition~4.1\textup{(vi)}]{Ozawa04}.
Since $\A_X$ is hyperfinite, \cite[Corollary~4.5]{BDJ08} gives Connes
embeddability of $\Pcal_X$. Apply \cite[Corollary~6.2]{Ozawa04}.
\end{proof}

\subsubsection{The stabilizer bundle}
Let \(H_x=\{g\in \FF_r:gx=x\}\) and define
\[
\mathcal B=\int_X^\oplus \mathcal B_x\,d\mu(x),\qquad \mathcal B_x=\L(H_x).
\]
This is the group von Neumann algebra of the isotropy group bundle. Inside \(\Pcal_X\), it is \(\A_X'\cap \Pcal_X\): a Fourier coefficient at \(g\) commutes with all functions precisely when it is supported on \(\operatorname{Fix}(g)\). Since \(\Pcal_X\) is finite and QWEP, the trace-preserving expectation onto \(\mathcal B\) shows that \(\mathcal B\) is QWEP. The diagonal algebra $\A_X$ is central in $\mathcal B$.

Enumerate \(\FF_r\) and choose measurably, for each \((x,y)\in\mathcal R\), an element \(t(x,y)\in \FF_r\) with \(t(x,y)y=x\), taking \(t(x,x)=e\). Write \(t_{xy}=t(x,y)\). Conjugation gives trace-preserving fiber isomorphisms
\[
\alpha_{xy}:\mathcal B_y\longrightarrow \mathcal B_x,\qquad
\alpha_{xy}(\lambda(h))=\lambda(t_{xy}ht_{xy}^{-1}).
\]
For \((x,y),(y,z)\in\mathcal R\), put
\[
\Omega(x,y,z)=\lambda(t_{xy}t_{yz}t_{xz}^{-1})\in\cU(\mathcal B_x).
\]
The element in the subscript fixes \(x\). Direct multiplication verifies
\begin{align*}
\alpha_{xy}\alpha_{yz}&=\Ad_{\Omega(x,y,z)}\alpha_{xz},\\
\Omega(x,y,z)\Omega(x,z,w)
&=\alpha_{xy}\bigl(\Omega(y,z,w)\bigr)\Omega(x,y,w).
\end{align*}
These are the normalized cocycle-action identities for the measured groupoid \(\mathcal R\). Its crossed product uses the fiber traces and the invariant probability measure \(\mu\); see \cite[Definition~2.2 and the subsequent crossed-product construction]{PSV20} for this framework.

\begin{theorem}\label{thm:groupoid}
There is a trace-preserving isomorphism
\[
\mathcal B\rtimes_{\alpha,\Omega}\mathcal R\cong \Pcal_X.
\]
Consequently \(\mathcal B\) and the twisted crossed product are QWEP, whereas \(\L(\mathcal R)\) is non-QWEP. This cocycle cannot be removed by a measurable field of coefficient unitaries. The same obstruction persists after tensoring each fiber with any fixed finite tracial QWEP algebra with separable predual.
\end{theorem}
\begin{proof}
Every arrow $g:y\to x$ has the unique decomposition
\[
g=h_xt_{xy},\qquad h_x=gt_{xy}^{-1}\in H_x.
\]
For composable arrows,
\[
(h_xt_{xy})(k_yt_{yz})
=h_x\bigl(t_{xy}k_yt_{xy}^{-1}\bigr)\bigl(t_{xy}t_{yz}t_{xz}^{-1}\bigr)t_{xz}.
\]
Thus the decomposition preserves convolution, involution and the
identity-arrow trace. Passing to the regular representations gives
$\mathcal B\rtimes_{\alpha,\Omega}\mathcal R\cong \Pcal_X$.

If coefficient unitaries removed $\Omega$, the perturbed action would
preserve the fiber traces. The untwisted crossed product would contain
$\L(\mathcal R)$ with trace-preserving expectation, contradicting
QWEP of $\Pcal_X$. After tensoring by $\mathcal D$, the same argument applies to
$\Pcal_X\barot \mathcal D$ and the product fiber traces.
\end{proof}

\medskip

\section{Statement on the use of generative AI}
During the development and preparation of this manuscript, the author used AI tools, including ChatGPT, to assist with language polishing, LaTeX editing, and exploratory mathematical discussions. Interactions with GPT-6 Astra also contributed some ideas to the proof-development process. The author critically assessed and reformulated these suggestions before incorporating them into the manuscript and independently verified all mathematical arguments presented. The author takes full responsibility for the mathematical content and accuracy of the final manuscript.

\section{Acknowledgements}
The author thanks Prof. Xiao Xiong and Prof. Quanhua Xu for their collaboration on
\emph{Harmonic analysis on twisted crossed products} \cite{xxz26} and for
discussions of the QWEP questions studied here.

The author was supported by the FWO Odysseus 1 grant G.0H94.18N: Analysis and
Partial Differential Equations, and the Methusalem programme of the Ghent
University Special Research Fund (BOF), Grant No.~01M01021.


\begin{thebibliography}{99}
\bibitem{ALT26} V. Alekseev, J. Liu and A. Thom, \emph{Spectral gap, internality of relative commutants, and non-hyperlinear groups}, preprint, 2026, arXiv:2609.37761v1.
\bibitem{BDJ08} N. P. Brown, K. J. Dykema and K. Jung, \emph{Free entropy dimension in amalgamated free products}, Proc. London Math. Soc. (3) \textbf{97} (2008), 339--367.
\bibitem{CF26} S. Chakraborty and F. Flores, \emph{Connes-embeddability and twisted group von Neumann algebras}, preprint, 2026, arXiv:2609.40246v1.
\bibitem{CXY12} Z. Chen, Q. Xu and Z. Yin, \emph{Harmonic analysis on quantum tori}, Comm. Math. Phys. \textbf{322} (2013), 755--805.
\bibitem{DV26} A. Dogon and I. Vigdorovich, \emph{Hyperlinearity, stability and asymptotic spectral gap of higher rank lattices}, preprint, 2026, arXiv:2506.20843v2.
\bibitem{Eckhardt25} C. Eckhardt, \emph{Residually finite amenable groups that are not Hilbert--Schmidt stable}, preprint, 2025, arXiv:2501.07791v2.
\bibitem{EJZ10} M. Ershov and A. Jaikin-Zapirain, \emph{Property \((T)\) for noncommutative universal lattices}, Invent. Math. \textbf{179} (2010), 303--347.
\bibitem{Gonz} A. M. Gonz\'alez-P\'erez, \emph{Crossed-product extensions of $L_p$-bounds for amenable actions}, J. Funct. Anal. \textbf{274} (2018), 2846--2883.
\bibitem{Junge04} M. Junge, \emph{Fubini's theorem for ultraproducts of noncommutative $L_p$-spaces}, Canad. J. Math. \textbf{56} (2004), 983--1021.
\bibitem{KN06} M. Kassabov and N. Nikolov, \emph{Universal lattices and property $\tau$}, Invent. Math. \textbf{165} (2006), 209--224.
\bibitem{KS09} M. Kassabov and M. Sapir, \emph{Nonlinearity of matrix groups}, J. Topol. Anal. \textbf{1} (2009), 251--260.
\bibitem{KT26} G. Kun and A. Thom, \emph{Nonsofic wreath products of residually finite groups}, preprint, 2026, arXiv:2608.06222v3.
\bibitem{KLPV25} B. Kunyavski\u\i, A. Lavrenov, E. Plotkin and N. Vavilov, \emph{Bounded generation of Steinberg groups over Dedekind rings of arithmetic type}, Eur. J. Math. \textbf{11} (2025), Article~23.
\bibitem{Manzoor26} A. Manzoor, \emph{There is an equivalence relation whose von Neumann algebra is not Connes embeddable}, preprint, arXiv:2502.06697v2.
\bibitem{NR11} S. Neuwirth and \'E. Ricard, \emph{Transfer of Fourier multipliers into Schur multipliers and sumsets in a discrete group}, Canad. J. Math. \textbf{63} (2011), 1161--1187.
\bibitem{Ozawa04} N. Ozawa, \emph{About the QWEP conjecture}, Internat. J. Math. \textbf{15} (2004), 501--530.
\bibitem{Pestov08} V. G. Pestov, \emph{Hyperlinear and sofic groups: a brief guide}, Bull. Symbolic Logic \textbf{14} (2008), 449--480.
\bibitem{Peterson10} J. Peterson, answer to \emph{Does a crossed product of the hyperfinite factor of type II$_1$ and a free group have QWEP?}, MathOverflow, 10 July 2010. \href{https://mathoverflow.net/questions/31278}{Question 31278}.
\bibitem{PSV20} S. Popa, D. Shlyakhtenko and S. Vaes, \emph{Classification of regular subalgebras of the hyperfinite II$_1$ factor}, J. Math. Pures Appl. \textbf{140} (2020), 280--308.
\bibitem{Serre80} J.-P. Serre, \emph{Trees}, translated by J. Stillwell, Springer-Verlag, Berlin--New York, 1980.
\bibitem{Sug21} R. Sugawara, \emph{Universal central extensions of linear groups over rings of non-commutative Laurent polynomials, associated $K_1$-groups and $K_2$-groups}, Tsukuba J. Math. \textbf{45} (2021), 13--36.
\bibitem{Thom10} A. Thom, \emph{Examples of hyperlinear groups without factorization property}, Groups Geom. Dyn. \textbf{4} (2010), 195--208.
\bibitem{Thom26} A. Thom, \emph{A conditional construction of a nonhyperlinear group and the centralizer problem}, preprint, 2026, arXiv:2609.27795v1.
\bibitem{xxz26} X. Xiong, Q. Xu, K. Zeng, \emph{Harmonic analysis on twisted crossed products}, work in preparation.
\bibitem{Zeng23} K. Zeng, \emph{Some problems in harmonic analysis on twisted crossed products}, doctoral thesis, Universit\'e de Bourgogne Franche-Comt\'e, 2023.
\bibitem{Zeng26} K. Zeng, \emph{Harmonic Analysis on Twisted Crossed Products}, in M. Ruzhansky and K. Van Bockstal (eds.), \emph{Ghent Analysis and PDE Seminar}, Trends in Mathematics \textbf{18}, Birkh\"auser, Cham, 2026, 107--116.
\end{thebibliography}
\end{document}